\documentclass[12pt, draftclsnofoot, onecolumn]{ieeeconf}

\usepackage{amsmath,amssymb,graphicx,paralist,subfigure,pdfpages,cite,algorithm,algpseudocode,paralist}
\newtheorem{lem}{Lemma} 
\newtheorem{theorem}{Theorem}

\newtheorem{assump}{Assumption}

\def\ln{{\rm ln}}

\def\mc{\mathcal}
\def\mb{\mathbf}
\def\mbb{\mathbb}
\def\ra{\rightarrow}

\IEEEoverridecommandlockouts                             
\usepackage{comment}

\def\mbb{\mathbb}%R
\def\mb{\mathbf}%vector
\def\mc{\mathcal}%set

\def\ul{\underline}

\def\bs{\boldsymbol}
\def\ol{\overline}
\newcommand{\mn}[1]{{\left\vert\kern-0.25ex\left\vert\kern-0.25ex\left\vert\kern0.3ex #1 
    \kern0.3ex\right\vert\kern-0.25ex\right\vert\kern-0.25ex\right\vert}}

\begin{document}
	\title{A linear algorithm for optimization over directed graphs with geometric convergence} % Title
	\author{Ran Xin,~\emph{Student Member,~IEEE},  and Usman A. Khan,~\emph{Senior Member,~IEEE}
\thanks{R.~Xin and U.~A.~Khan are with the Department of Electrical and Computer Engineering, Tufts University, 161 College Ave, Medford, MA 02155; {\texttt{ran.xin@tufts.edu, khan@ece.tufts.edu}}. This work has been partially supported by an NSF Career Award \# CCF-1350264.}}

\maketitle

\begin{abstract}
In this letter, we study distributed optimization, where a network of agents, abstracted as a directed graph, collaborates to minimize the average of locally-known convex functions. Most of the existing approaches over directed graphs are based on push-sum (type) techniques, which use an independent algorithm to asymptotically learn either the left or right eigenvector of the underlying weight matrices. This strategy causes additional computation, communication, and nonlinearity in the algorithm. In contrast, we propose a linear algorithm based on an inexact gradient method and a gradient estimation technique. Under the assumptions that each local function is strongly-convex with Lipschitz-continuous gradients, we show that the proposed algorithm geometrically converges to the global minimizer with a sufficiently small step-size. We present simulations to illustrate the theoretical findings.

{\color{black}\keywords 		Distributed optimization, directed graphs}
\end{abstract}
	
\section{Introduction}\label{s1}
In this letter, we consider distributed optimization over multi-agent networks. Formally, each agent~$i$ has access only to a private function,~$f_i:\mbb{R}^p\rightarrow\mbb{R}$. The goal is to minimize the average of {\color{black}these functions},~$\frac{1}{n}\sum_{i=1}^nf_i(\mb{x})$, via information exchange among the agents. We focus on the case where the communication network is described by {\color{black} an arbitrary \emph{directed} graph}. Early work on distributed optimization {\color{black}includes distributed sub-gradient descent (DGD)~\cite{uc_Nedic},} which converges to the optimal solution at a sublinear rate, i.e.,~$O(\frac{\ln k}{\sqrt{k}})$ for arbitrary (possibly non-differentiable) convex functions and~$O(\frac{\ln k}{k})$ for strongly-convex functions, where~$k$ is the number of iterations. These methods are slow due to the diminishing step-sizes. With the help of strong-convexity and Liptschiz-continuous gradients, algorithms with faster convergence rates have been developed. In particular, DGD with a constant step-size~\cite{DGD_Yuan} converges geometrically to an error ball around the optimal solution. Another method, EXTRA~\cite{EXTRA}, achieves geometric convergence to the global optimal solution with the requirement of symmetric weights. Of relevance are Refs.~\cite{xu2015augmented,Augmented_EXTRA,xu2018convergence,GQu_nesterov}, which combine inexact gradient methods and a gradient estimation technique based on dynamic average consensus~\cite{zhu2010discrete}. Additional related work and applications can be found in~\cite{6119236,jakovetic2017unification,RajaBajwa.ITSP16,8123915,8264076,YING2018253}.

All of the aforementioned methods require the underlying graphs to be undirected or weight-balanced. This requirement, however, may not be practical, for example, when the agents broadcast at different power levels {\color{black}leading to communication capability in one direction but not in the other.} It is natural thus to develop optimization and learning algorithms that are applicable to directed graphs. The primary challenge in dealing with directed graphs is that it may not be possible to construct doubly-stochastic weight matrices for information fusion. The weighted adjacency matrix for directed graphs, in general, may only be either row-stochastic or column-stochastic, but not both. See{\color{black}~\cite{gharesifard2012distributed}} for work on balancing the weights in strongly-connected directed graphs. 

{\color{black}The existing approaches for optimization over directed graphs are motivated by combining average-consensus methods developed for directed graphs with optimization algorithms designed for undirected graphs.} For instance, subgradient-push introduced in~\cite{opdirect_Tsianous} and further studied in~\cite{opdirect_Nedic} combines push-sum consensus~\cite{ac_directed0} and DGD; {\color{black}A linear algorithm over directed graphs, called Directed-Distributed Gradient Descent (D-DGD)}, was introduced in~\cite{D-DGD,D-DPS}, and is based on surplus consensus~\cite{ac_Cai1} and DGD. Such DGD-based methods, however, restricted by the diminishing step-size, converge relatively slowly at~$O(\frac{\ln k}{\sqrt{k}})$ for general convex functions and~$O(\frac{\ln k}{k})$ for strongly-convex functions. The convergence rate has been recently improved in DEXTRA~\cite{DEXTRA}, which converges geometrically to the global optimal given that its step-size lies in an interval and the objective functions are strongly-convex with Lipschitz-continuous gradients. DEXTRA was subsequently improved in ADD-OPT/Push-DIGing~\cite{add-opt,opdirect_nedicLinear}, which geometrically converges with a sufficiently small step-size. The implementation of DEXTRA and {\color{black}ADD-OPT/Push-DIGing} requires each agent to know its out-degree in order to construct a column-stochastic weight matrix. This requirement is later removed in~\cite{linear_row} and FROST~\cite{xin2018fast}, which use row-stochastic weights and {\color{black}thus require no knowledge of out-degrees as each agent locally decides weights assigned to the incoming information.} What is common among these fast methods over directed graphs is that they all are based on {\color{black}push-sum (type) techniques, which make the resulting algorithm nonlinear because an independent algorithm} is used to asymptotically learn either the right or the left eigenvector, corresponding to the eigenvalue of~$1$, of the weight matrix. {\color{black}This strategy causes additional computation and communication on the agents.}

In this paper, we provide  a \textit{linear} distributed optimization algorithm that converges geometrically to the global optimal with a sufficiently small step-size and when the objective functions are strongly-convex with Lipschitz-continuous gradients. In the rest of the paper,  Section~\ref{s2} provides the algorithm development and its relationship with existing approaches, while Section~\ref{s3} details the convergence analysis. Section~\ref{s4} presents numerical experiments and Section~\ref{s5} concludes the paper.

\textbf{Basic Notation:} We use lowercase bold letters to denote vectors and uppercase italic letters to denote matrices. The matrix,~$I_n$, represents the~$n\times n$ identity, whereas~$\mb{1}_n$ is the~$n$-dimensional column vector of all~$1$'s. For an arbitrary vector,~$\mb{x}$, we denote its~$i$th element  by~$[\mb{x}]_i$. We denote by~$X\otimes Y$, the Kronecker product of two matrices,~$X$ and~$Y$. {\color{black}For a matrix,~$X$, we denote~$\rho(X)$ as its spectral radius and~$X_\infty$ as its infinite power (if it exists), i.e.,~$X_\infty=\lim_{k\ra\infty}X^k$.} For a primitive, row-stochastic matrix,~$\ul{A}$, we denote its left and right eigenvectors corresponding to the eigenvalue of~$1$ by~$\bs{\pi}_r$ and~$\mb{1}_n$, respectively, such that~$\bs{\pi}_r^\top\mb{1}_n = 1$. Similarly, for a primitive, column-stochastic matrix,~$\ul{B}$, we denote its left and right eigenvectors corresponding to the eigenvalue of~$1$ by~$\mb{1}_n$ and~$\bs{\pi}_c$, respectively, such that~$\mb{1}_n^\top\bs{\pi}_c = 1$. The notation~$\|\cdot\|_2$ denotes the Euclidean norm of vectors and~$\mn{\cdot}_2$ denotes the spectral norm of matrices.

\section{Algorithm Development}\label{s2}
In this section, we mathematically formulate the optimization problem and describe the proposed algorithm and its relationship with the existing methods. Consider a network of~$n$ agents whose communication links are described by a strongly-connected directed graph,~$\mc{G}=(\mc{V},\mc{E})$, where~$\mc{V}$ is the index set of agents, and~$\mc{E}$ is the collection of ordered pairs,~$(i,j),i,j\in\mc{V}$, such that agent~$j$ can send information to agent~$i$, i.e.,~$j\rightarrow i$. We define~$\mc{N}_i^{{\scriptsize \mbox{in}}}$ as the collection of in-neighbors, i.e., the set of agents that can send information to agent~$i$. Similarly,~$\mc{N}_i^{{\scriptsize \mbox{out}}}$ is the set of out-neighbors of agent~$i$. Note that both~$\mc{N}_i^{{\scriptsize \mbox{in}}}$ and~$\mc{N}_i^{{\scriptsize \mbox{out}}}$ include node~$i$. We assume that each agent~$i$ knows\footnote{Such an assumption is standard in the related literature, see, e.g.,~\cite{opdirect_Nedic,opdirect_Tsianous,ac_Cai1,D-DGD,D-DPS,DEXTRA,add-opt}.} its out-degree (the number of out-neighbors), denoted by~$|\mc{N}_i^{{\scriptsize \mbox{out}}}|$; see~\cite{bullo_book} for details. 

We focus on solving a convex optimization problem distributed over the above multi-agent network. In particular, the network of agents cooperatively solves the following:
\begin{align}
\mbox{P1}:
\quad\mbox{min  }&F(\mb{x})=\frac{1}{n}\sum_{i=1}^nf_i(\mb{x}),\nonumber
\end{align}
where each~$f_i:\mbb{R}^p\rightarrow\mbb{R}$ is known only to agent~$i$. We assume that each local function,~$f_i(\mb{x})$, is strongly-convex and has Lipschitz-continuous gradients. Our goal is to design a distributed algorithm such that the {\color{black}iterates at each agent converge} to the global optimal solution of Problem P1 via information exchange with nearby agents over the directed graph,~$\mc{G}$. We formalize the set of assumptions as follows. 
\begin{assump}\label{asp1}
	The  graph,~$\mc{G}$, is strongly-connected and each agent in the network knows its out-degree.
\end{assump}

\begin{assump}\label{asp2}
	Each local function,~$f_i$, is strongly-convex, and has globally Lipschitz-continuous gradient, i.e., for any~$i$ and~$\mb{x}_1, \mb{x}_2\in\mbb{R}^p$,
	\begin{enumerate}[(i)]
		\item there exists a positive constant~$\beta$ such that 	
		$$\qquad\|\mb{\nabla} f_i(\mb{x}_1)-\mb{\nabla} f_i(\mb{x}_2)\|_2\leq \beta\|\mb{x}_1-\mb{x}_2\|_2;$$
		\item there exists a positive constant~$\alpha$ such that		
		$$f_i(\mb{x}_1)-f_i(\mb{x}_2)\leq\mb{\nabla} f_i(\mb{x}_1)^\top(\mb{x}_1-\mb{x}_2)-\frac{\alpha}{2}\|\mb{x}_1-\mb{x}_2\|_2^2.$$
	\end{enumerate}
	Clearly, the Lipschitz-continuity and strongly-convexity constants for the global objective function~$F(\mb{x})$ are~$\beta$ and~$\alpha$, respectively. Assumption~\ref{asp2} ensures that the optimal solution, {\color{black}denoted as~$\mb{x}^*$}, for P1 exists and is unique.
\end{assump}

\textbf{Algorithm description:} To solve Problem P1, we propose the following algorithm. Each agent,~$i\in\mc{V}$, maintains two variables:~$\mb{x}_{i}(k)$,~$\mb{y}_{i}(k)\in\mbb{R}^p$, where $k$ is discrete-time index. The algorithm, initialized with~$\mb{y}_i(0)=\nabla f_i(\mb{x}_i(0))$ and {\color{black}with arbitrary~$\mb{x}_i(0),\forall i$}, performs the following iterations. 
\begin{subequations}\label{alg1}
	\begin{align}
	\mb{x}_i(k+1)=&{\color{black}\sum_{j=1}^{n}}a_{ij}\mb{x}_{j}(k)-\eta\mb{y}_i(k),\label{alg1a}\\
	\mb{y}_i(k+1)=&{\color{black}\sum_{j=1}^{n}}b_{ij}\Big(\mb{y}_j(k)+\nabla f_j\big(\mb{x}_j(k+1)\big)-\nabla f_j\big(\mb{x}_j(k)\big)\Big),\label{alg1d}
	\end{align}
\end{subequations}
where the step-size,~$\eta$, is some positive constant. The weights,~$a_{ij}$'s and~$b_{ij}$'s satisfy the following conditions:
\begin{align} 
a_{ij}&=\left\{
\begin{array}{rl}
>0,&j\in\mc{N}_i^{{\scriptsize \mbox{in}}},\\
0,&\mbox{otherwise},
\end{array}
\right.
\quad
\sum_{j=1}^na_{ij}=1,\forall i\begin{color}{black},\end{color} \label{a}
\end{align} 
\begin{align} 
b_{ij}&=\left\{
\begin{array}{rl}
>0,&i\in\mc{N}_j^{{\scriptsize \mbox{out}}},\\
0,&\mbox{otherwise},
\end{array}
\right.
\quad
\sum_{i=1}^nb_{ij}=1,\forall j. \label{b}
\end{align} 
Eq.~\eqref{a} leads to a row-stochastic matrix~$\ul{A}=\{a_{ij}\}$, which is easy to implement as each agent locally decides the weights. Eq.~\eqref{b}, on the other hand, results in a column-stochastic matrix~$\ul{B}=\{b_{ij}\}$, whose distributed implementation only requires each agent to know its out-degree. In particular, we can construct such weights as~$b_{ij}=1/|\mc{N}_j^{{\scriptsize \mbox{out}}}|,\forall i,j$.

The algorithm in Eqs.~\eqref{alg1} can be explained as follows. To implement Eq.~\eqref{alg1a}, the receiving agent~$i$ decides on the weights~$a_{ij}$ assigned to the incoming~$\mb{x}_j(k)$'s such that~$a_{ij}$'s sum to~$1$. Implementation of Eq.~\eqref{alg1d} requires the sending agent to scale the transmission~$\mb{y}_j(k)+\nabla f_j\big(\mb{x}_j(k+1)\big)-\nabla f_j\big(\mb{x}_j(k)$ by appropriate choice of~$b_{ij}$'s (to ensure column-stochasticity of~$\ul{B}$) as the out-degree of agent~$j$ may not be known to agent~$i$. Agent~$i$ subsequently adds these received messages to implement Eq.~\eqref{alg1d}. Intuitively, Eq.~\eqref{alg1d} asymptotically learns the average,~$\frac{1}{n}\sum_{i=1}^{n}\nabla f_i(\mb{x}_i(k))$, of the local gradients,{\color{black}~\cite{xu2015augmented,xu2018convergence,Augmented_EXTRA,GQu_nesterov,zhu2010discrete}}; and thus~Eq.~\eqref{alg1a} approaches a centralized gradient descent, as the descent direction,~$\mb{y}_i(k)$, becomes the gradient of the global objective function over time.

\textbf{Relation with existing work:} We now briefly compare the proposed algorithm with existing techniques. {\color{black}The algorithms in Refs.~\cite{Augmented_EXTRA, xu2015augmented,xu2018convergence},} can be summarized as a single class of algorithms over undirected graphs with the following form:
\begin{subequations}\label{alg2}
	\begin{align}
	\mb{x}_i(k+1)=&\sum_{j=1}^{n}w_{ij}\mb{x}_{j}(k)-\eta\mb{y}_{i}(k),\label{alg2a}\\
	\mb{y}_i(k+1)=&\sum_{j=1}^{n}w_{ij}\mb{y}_j(k)+\nabla f_i\big(\mb{x}_i(k+1)\big)-\nabla f_i\big(\mb{x}_i(k)\big),\label{alg2d}
	\end{align}
\end{subequations}
where~$W=\{w_{ij}\}$ is doubly-stochastic. It is shown in Ref.~\cite{Augmented_EXTRA,xu2018convergence}, that Eqs.~\eqref{alg2} converge geometrically to the optimal solution of Problem P1 as long as the step-size,~$\eta$, is sufficiently small. This algorithm, however, is not applicable to directed graphs as it may not be possible to construct doubly-stochastic weights. 

To overcome this issue, Refs.~\cite{add-opt,opdirect_nedicLinear,linear_row,xin2018fast} leverage push-sum (type) techniques, with either row- or column-stochastic weights, towards the algorithm in Eqs.~\eqref{alg2}. Refs.~\cite{linear_row,xin2018fast}, e.g., propose the following algorithm:
\begin{align*}
\mb{y}_i(k+1)=&\sum_{j=1}^na_{ij}\mb{y}_i(k),\\
\mb{x}_i(k+1)=&\sum_{j=1}^na_{ij}\mb{x}_i(k)-\begin{color}{black}\eta_i\end{color}\mb{z}_i(k),\\
\mb{z}_i(k+1)=&\sum_{j=1}^na_{ij}\mb{z}_i(k) +\frac{\nabla f_i\big(\mb{x}_i(k+1)\big)}{[\mb{y}_i(k+1)]_i}-\frac{\nabla f_i\big(\mb{x}_i(k)\big)}{[\mb{y}_i(k)]_i},
\end{align*}
{\color{black}where $\ul{A}=\{a_{ij}\}$ is row-stochastic}. Note that the first equation is an independent algorithm, which asymptotically learns the left eigenvector, corresponding to the eigenvalue of~$1$, of~$\ul{A}$. However, it adds nonlinearity to the overall algorithm along with additional computation and communication costs in contrast to the proposed algorithm in Eqs.~\eqref{alg1}. 

\textbf{Remarks:} \begin{color}{black}The algorithm,  Eqs.~\eqref{alg1}, proposed in this letter can be viewed as related to Eq.~\eqref{alg2} but without doubly-stochastic weights\end{color}, due to which we lose the nice eigenstructure within the weight matrices. It is rather straightforward to notice that a linear extension of Eqs.~\eqref{alg2} to the directed graphs is non-trivial as all earlier attempts were made by adding nonlinearity to the original set of equations. One of the major challenges lies in the fact that even though the contraction of a doubly-stochastic~$W$ is well-established in the subspace orthogonal to~$\mb{1}_n$, it is not straightforward to establish simultaneous contractions for a row-stochastic matrix,~$\ul{A}$, and a column-stochastic matrix,~$\ul{B}$. The latter requires working with arbitrary norms (as opposed to the~$2$-norm applicable to doubly-stochastic matrices) and norm-equivalence constants, as we  show in Lemma~\ref{lem1} and onwards. 
 
\section{Convergence Analysis}\label{s3}
For the sake of analysis, we now write Eqs.~\eqref{alg1} in matrix form. The variables~$\mb{x}(k)$ and~$\mb{y}(k)$ collect all the local variables~$\mb{x}_i(k)$'s and~$\mb{y}_i(k)$'s in a vector, respectively, and 
\begin{eqnarray}\label{not1}
\nabla\mb{f}(k)=
\left[
\begin{array}{c}
\nabla {f}_1\big(\mb{x}_{1}(k)\big)\\
\vdots\\
\nabla {f}_n\big(\mb{x}_{n}(k)\big)
\end{array}
\right]\in \mathbb{R}^{np}.
\end{eqnarray}
Let~$A=\ul{A}\otimes I_p$ and~$B=\ul{B}\otimes I_p$, where $\otimes$ is the Kronecker product. {\color{black}We denote~$\mb{x}^*$ as the optimal solution of Problem P1.} We now rewrite Eqs.~\eqref{alg1} in a compact matrix form as follows:
\begin{subequations}\label{alg1_matrix}
	\begin{align}
	\mb{x}(k+1)=&A\mb{x}(k)-\begin{color}{black}\eta\end{color} \mb{y}(k),\label{alg1_ma}\\
	\mb{y}(k+1)=&B\Big(\mb{y}(k)+\nabla \mb{f}(k+1)-\nabla \mb{f}(k)\Big),\label{alg1_mb}
	\end{align}
\end{subequations}
where~$\mb{y}(0)=\nabla\mb{f}(0)$ and~$\mb{x}(0)$ is arbitrary.

\subsection{Auxiliary relations}
We next start the convergence analysis with a key lemma regarding the contraction in consensus process with row- and column-stochastic weight matrices, respectively. 
\begin{lem}\label{lem1}
	Consider the weight matrices~$A=\ul{A}\otimes I_p$ and~$B=\ul{B}\otimes I_p$. Then there exist vector norms,~$\|\cdot\|_A$ and~$\|\cdot\|_B$, such that for all~$\mb{a}\in\mbb{R}^{np}$,
	\begin{align}\label{A_ctr}
	\left\|A\mb{a}-A_\infty\mb{a}\right\|_A&\leq\sigma_A\left\|\mb{a}-A_\infty\mb{a}\right\|_A,\\\label{B_ctr}
	\left\|B\mb{a}-B_\infty\mb{a}\right\|_B&\leq\sigma_B\left\|\mb{a}-B_\infty\mb{a}\right\|_B,
	\end{align}
	where~$0<\sigma_A<1$ and~$0<\sigma_B<1$ are some constants.  
\end{lem}
\begin{proof}
Since~$\ul{A}$ is irreducible, row-stochastic with positive diagonals, from Perron-Frobenius theorem we have that~$\rho(\ul{A})=1$, every eigenvalue of~$\ul{A}$ other than~$1$ is strictly less than~$\rho(\ul{A})$, and~$\bs{\pi}_r^\top$ is a strictly positive left eigenvector corresponding to the eigenvalue of~$1$ with~$\mb{1}_n^\top\bs{\pi}_r = 1$; thus~$\lim_{k\rightarrow\infty} \ul{A}^k = \mb{1}_n\bs{\pi}_r^\top$. We further have 
\begin{align*}
A_{\infty}=\lim_{k\rightarrow\infty}{A^k}=
%\lim_{k\rightarrow\infty}{(\ul{A}\otimes I_p)^k}=
\left(\lim_{k\rightarrow\infty}{\ul{A}^k}\right)\otimes I_p=\left(\mb{1}_n\bs{\pi}_r^\top\right)\otimes I_p.
\end{align*} 
It follows that
	\begin{align}
	AA_{\infty} &= (\ul{A}\otimes I_p)\Big((\mb{1}_n\bs{\pi}_r^\top)\otimes I_p\Big) = A_{\infty}, \nonumber \\
	A_{\infty}A_{\infty} &= \Big((\mb{1}_n\bs{\pi}_r^\top)\otimes I_p\Big)\Big((\mb{1}_n\bs{\pi}_r^\top)\otimes I_p\Big)=A_{\infty}. \nonumber
	\end{align}
	Thus~$AA_{\infty}-A_{\infty}A_{\infty}$ is a zero matrix, which leads to the following relation:
	\begin{eqnarray}\label{eq1}
	A\mb{a}-A_\infty \mb{a}=(A-A_{\infty})(\mb{a}-A_{\infty}\mb{a}).
	\end{eqnarray}
\begin{color}{black}Since~$\rho(A-A_{\infty})=\rho((\ul{A}-\mb{1}_n\bs{\pi}_r^\top)\otimes I_p)<1,$ we have from Lemma 5.6.10 in~\cite{hornjohnson:13} that there exists a matrix norm, say~$\mn{\cdot}_A$, such that
\begin{align}
\sigma_A\triangleq\mn{A-A_{\infty}}_A<1.
\end{align}
Moreover, from Theorem 5.7.13 in~\cite{hornjohnson:13}, we know that for any matrix norm,~$\mn{\cdot}_A$, there exists a compatible vector norm, say~$\|\cdot\|_A$, such that~$\|X\mb{x}\|_A\leq\mn{X}_A\|\mb{x}\|_A$, for all matrices,~$X$, and all vectors,~$\mb{x}$; hence, Eq.~\eqref{eq1} leads~to
\vspace{-0.1cm}
\begin{eqnarray*}
\|{A\mb{a}-A_\infty \mb{a}}\|_A&=&\|{(A-A_{\infty})(\mb{a}-A_{\infty}\mb{a})}\|_A,\\ &\leq&\mn{A-A_\infty}_A\|\mb{a}-A_{\infty}\mb{a}\|_A,\\
&=& \sigma_A\|\mb{a}-A_{\infty}\mb{a}\|_A,
\end{eqnarray*}
and Eq.~\eqref{A_ctr} follows. Similarly, Eq.~\eqref{B_ctr} follows for some matrix norm,~$\mn{\cdot}_B$, with~$\sigma_B\triangleq\mn{B-B_{\infty}}_B$. \end{color}
\end{proof}

The following lemma is a direct consequence of the column-stochasticity of~$\ul{B}$ and {\color{black}the initial condition that~$\mb{y}(0)=\nabla\mb{f}(0)$.}
\begin{lem}\label{sum_equ}
	We have~$(\mb{1}_n^\top \otimes I_p) \mb{y}(k) = (\mb{1}_n^\top \otimes I_p) \nabla\mb{f}(k),\forall k$.
\end{lem}
\begin{proof}
	Recall Eq.~\eqref{alg1_mb} and multiply both sides of Eq.~\eqref{alg1_mb} with~$\mb{1}_n^\top \otimes I_p$. We get
	\begin{align}
			&(\mb{1}_n^\top \otimes I_p)\mb{y}(k+1)\nonumber\\
			&=(\mb{1}_n^\top \otimes I_p)(\ul{B}\otimes I_p)\Big(\mb{y}(k)+\nabla \mb{f}(k+1)-\nabla \mb{f}(k)\Big)	\nonumber\\
			&= (\mb{1}_n^\top \otimes I_p)\mb{y}(k) + (\mb{1}_n^\top \otimes I_p)\nabla \mb{f}(k+1)
			- (\mb{1}_n^\top \otimes I_p)\nabla \mb{f}(k)\nonumber\\
			&= (\mb{1}_n^\top \otimes I_p)\Big(\mb{y}(0)-\nabla\mb{f}(0)\Big)+(\mb{1}_n^\top \otimes I_p)\nabla\mb{f}(k+1)\nonumber\\
			&= (\mb{1}_n^\top \otimes I_p)\nabla\mb{f}(k+1), \nonumber
	\end{align}
	which completes the proof.
\end{proof}

\newpage
Lemma~\ref{sum_equ} shows that the average of~$\mb{y}_i(k)$'s preserves the average of local gradients. The next lemma, a standard result in convex optimization theory from~\cite{opt_literature0,Augmented_EXTRA}, states that the distance to the optimal minimizer shrinks by at least a fixed ratio if we perform a gradient descent step.
\begin{lem}\label{centr_d}
Suppose that~$g:\mbb{R}^p\rightarrow\mbb{R}$ is strongly convex with Lipschitz-continuous gradient. Let~$\alpha$ and~$\beta$ be its strong-convexity and Lipschitz-continuity constants respectively. For~$\forall \mb{x}\in\mbb{R}^p$ and~$0<\theta<\frac{2}{\beta}$, we have ~$$\left\|\mb{x}-\theta\nabla g(\mb{x})-\mb{x}^*\right\|_2\leq\tau\left\|\mb{x}-\mb{x}^*\right\|_2,$$ where~$\tau=\max\left(\left|1-\alpha \theta\right|,\left|1-\beta\theta \right|\right)$.
\end{lem}

The subsequent convergence analysis is based on deriving a contraction relationship in the proposed algorithm, i.e.,~$\|\mb{x}(k+1)-A_\infty\mb{x}(k+1)\|_A$,~$\|A_\infty\mb{x}(k+1)-\mb{1}_n \otimes \mb{x}^*\|_2$, and~$\|\mb{y}(k+1)-B_\infty\mb{y}(k+1)\|_B$, are bounded linearly by their values in the last iteration. We capture a relationship on these objects in the next lemmas. Before we proceed, note that all vector norms on finite-dimensional vector space are equivalent, i.e., there exist finite and positive constants,~$c,d,h,l,g,m$, such that:
\begin{align*}
\|\cdot\|_A &\leq c\|\cdot\|_B,~~\|\cdot\|_2 \leq h\|\cdot\|_B,~~\|\cdot\|_2 \leq g\|\cdot\|_A,\\
\|\cdot\|_B &\leq d\|\cdot\|_A,~~\|\cdot\|_B \leq l\|\cdot\|_2,~~\|\cdot\|_A \leq m\|\cdot\|_2.
\end{align*}

\begin{lem} \label{1}
The following inequality holds,~$\forall k$:
 \begin{align}
 \|\mb{x}&(k+1)-A_\infty\mb{x}(k+1)\|_A\nonumber\\ 
\leq&~\sigma_A\|\mb{x}(k)-A_\infty\mb{x}(k)\|_A + \eta m\mn{I_{np}-A_\infty}_2\left\|\mb{y}(k)\right\|_2
\nonumber
 \end{align}
\end{lem} 
\begin{proof}
Using Eq.~\eqref{alg1_ma} and Lemma.~\ref{lem1}, we have
\begin{align}
 \|\mb{x}&(k+1)-A_\infty\mb{x}(k+1)\|_A \nonumber\\
=&~\|A\mb{x}(k)-\eta \mb{y}(k)-A_{\infty}\Big(A\mb{x}(k)-\eta \mb{y}(k)\Big)\|_A, \nonumber\\
\leq&~ \sigma_A\|\mb{x}(k)-A_\infty\mb{x}(k)\|_A+\eta m\|\mb{y}(k)-A_\infty\mb{y}(k)\|_2,\nonumber\\
\leq&~\sigma_A\|\mb{x}(k)-A_\infty\mb{x}(k)\|_A + \eta m\mn{I_{np}-A_\infty}_2\left\|\mb{y}(k)\right\|_2
 \nonumber
\end{align}
and the lemma follows. 
\end{proof}
Next, we develop a relation for~$\|A_\infty\mb{x}(k+1)-\mb{1}_n \otimes \mb{x}^*\|_2$.

\begin{color}{black}\begin{lem} \label{2}
The following holds,~$\forall k$, when~$0<\eta<\frac{2}{n\beta\bs{\pi}_r^\top\bs{\pi}_c}$:
\begin{align}\label{lem5}
\|A&_\infty\mb{x}(k+1)-\mb{1}_n \otimes \mb{x}^*\|_2\nonumber\\
\leq&~\eta n\beta g(\bs{\pi}_r^\top\bs{\pi}_c)\|\mb{x}(k)-A_\infty\mb{x}(k)\|_A\nonumber\\
&+~\lambda\|A_\infty\mb{x}(k)-\mb{1}_n \otimes \mb{x}^*\|_2 +~\eta h\mn{A_\infty}_2\|\mb{y}(k)-B_{\infty}\mb{y}(k)\|_B,
\end{align}{\color{black}
where~$\lambda=\max\left(\left|1-\alpha n\eta(\bs{\pi}_r^\top\bs{\pi}_c)\right|,\left|1-\beta n\eta(\bs{\pi}_r^\top\bs{\pi}_c) \right|\right)$.}
\end{lem}
\end{color}
\begin{proof}
With~$A_\infty=(\mb{1}_n\bs{\pi}_r^\top)\otimes I_p=(\mb{1}_n\otimes I_p)(\bs{\pi}_r^\top\otimes I_p)$ and Eq.~\eqref{alg1_ma}, we have
\begin{align}\label{inte_1}
\|A&_\infty\mb{x}(k+1)-\mb{1}_n \otimes \mb{x}^*\|_2 \nonumber\\
=&\begin{color}{black}~\left\|A_{\infty}\Big(A\mb{x}(k)-\eta \mb{y}(k)+B_{\infty}\mb{y}(k)(-\eta+\eta)\Big)-\mb{1}_n \otimes \mb{x}^*\right\|_2, \nonumber\end{color}\\
\leq&~\left\|\big((\mb{1}_n\bs{\pi}_r^\top) \otimes I_p\big)\mb{x}(k)-(\mb{1}_n\otimes I_p)\mb{x}^*-\eta A_{\infty}B_{\infty}\mb{y}(k)\right\|_2 \nonumber\\
&+\eta h\mn{A_\infty}_2\left\|\mb{y}(k)-B_{\infty}\mb{y}(k)\right\|_B.
\end{align}Since the last term above matches with the last term in Eq.~\eqref{lem5}, what is left is to manipulate the first term. Before we proceed, define~$\nabla F(k)=\nabla F\big((\bs{\pi}_r^\top\otimes I_p)\mb{x}(k)\big)$, which is the global gradient evaluated at~$(\bs{\pi}_r^\top\otimes I_p)\mb{x}(k)$. Note that
\begin{align*}
A_\infty B_\infty=(\mb{1}_n\bs{\pi}_r^\top\otimes I_p)(\bs{\pi}_c\mb{1}_n^\top\otimes I_p)=\bs{\pi}_r^\top\bs{\pi}_c(\mb{1}_n\mb{1}_n^\top\otimes I_p).
\end{align*}
We have the following:
{\color{black}\begin{align*} 
\|\big(&(\mb{1}_n\bs{\pi}_r^\top) \otimes I_p\big)\mb{x}(k)-(\mb{1}_n\otimes I_p)\mb{x}^*-\eta A_{\infty}B_{\infty}\mb{y}(k)\|_2\\
%\leq&~\|(\mb{1}_n \otimes I_p)\Big((\bs{\pi}_r^\top\otimes I_p)\mb{x}(k)-\mb{x}^*+(1-1)n\eta(\bs{\pi}_r^\top\bs{\pi}_c)\nabla f\big((\bs{\pi}_r^\top\otimes I_p)\mb{x}(k)\big)\Big) -\eta \bs{\pi}_r^\top\bs{\pi}_c(\mb{1}_n \otimes I_p)(\mb{1}_n^\top\otimes I_p)\mb{y}(k)\|_2, \nonumber\\
\leq&~\left\|(\mb{1}_n \otimes I_p)\Big((\bs{\pi}_r^\top\otimes I_p)\mb{x}(k)-\mb{x}^*-n\eta(\bs{\pi}_r^\top\bs{\pi}_c)\nabla F(k)\Big)\right\|_2\\  &+\eta(\bs{\pi}_r^\top\bs{\pi}_c)\left\|n(\mb{1}_n\otimes I_p)\nabla F(k)- (\mb{1}_n, \otimes I_p)(\mb{1}_n^\top\otimes I_p)\mb{y}(k)\right\|_2,\\
:=&~s_1 + \eta s_2.
\end{align*}}From Lemma~\ref{centr_d}, we have that if~$0<\eta<2/(n\beta\bs{\pi}_r^\top\bs{\pi}_c)$,
$$s_1\leq\lambda\|A_\infty\mb{x}(k)-\mb{1}_n \otimes \mb{x}^*\|_2.$$
Recall that~$(\mb{1}_n^\top \otimes I_p) \mb{y}(k) = (\mb{1}_n^\top \otimes I_p) \nabla\mb{f}(k),\forall k,$ from Lemma~\ref{sum_equ}, we have
	\begin{align*}
		s_2 \leq& n\beta g(\bs{\pi}_r^\top\bs{\pi}_c)\|\mb{x}(k)-A_\infty\mb{x}(k)\|_A.
	\end{align*}
The lemma follows by using the above bounds in Eq.~\eqref{inte_1}.
\end{proof}
%\begin{proof}
%Recalling that $A_{\infty}=(\mb{1}_n \otimes I_p)(\bs{\pi}_r^\top\otimes I_p)$ and Eq.~\eqref{alg1_ma}, we have
%	\begin{align}\label{inte_1}
%		\|A&_\infty\mb{x}(k+1)-\mb{1}_n \otimes \mb{x}^*\|_2 \nonumber\\
%		=&~ \|A_{\infty}\Big(A\mb{x}(k)-\eta \mb{y}(k)\Big)-\mb{1}_n \otimes \mb{x}^*\|_2, \nonumber\\
%		=&~ {\small\|(\mb{1}_n \otimes I_p)(\bs{\pi}_r^\top\otimes I_p)\mb{x}(k)-A_{\infty}\eta\mb{y}(k)-(\mb{1}_n\otimes I_p)\mb{x}^*\|_2},
%		\nonumber\\
%		\leq&~ {\small\|(\mb{1}_n \otimes I_p)((\bs{\pi}_r^\top\otimes I_p)\mb{x}(k)-\mb{x}^*-\eta\nabla f((\bs{\pi}_r^\top\otimes I_p)\mb{x}(k)))\|_2} \nonumber\\
%		&+~\eta\|(\mb{1}_n \otimes I_p)\nabla f\big((\bs{\pi}_r^\top\otimes I_p)\mb{x}(k)\big)-A_{\infty}\mb{y}(k)\|_2, \nonumber\\
%		:=&~s_1 + \eta s_2.
%	\end{align}
%From Lemma~\ref{centr_d},~$s_1\leq\lambda\|A_\infty\mb{x}(k)-\mb{1}_n \otimes \mb{x}^*\|_2,$
%	and
%	\begin{align*}
%		s_2 \leq& \|(\mb{1}_n \otimes I_p)\nabla f\big((\bs{\pi}_r^\top\otimes I_p)\mb{x}(k)\big)\nonumber\\
%		&-\frac{1}{n}(\mb{1}_n\otimes I_p)(\mb{1}_n^\top\otimes I_p)\nabla \mb{f}(k)\|_2 \nonumber\\
%		&+\|\frac{1}{n}(\mb{1}_n\otimes I_p)(\mb{1}_n^\top\otimes I_p)\nabla \mb{f}(k)-A_{\infty}\mb{y}(k)\|_2, \nonumber\\
%		\leq&\beta g\|\mb{x}(k)-A_\infty\mb{x}(k)\|_A + \|\frac{1}{n}\mb{1}_n\mb{1}_n^\top-A_{\infty}\|_2
%		\|\mb{y}(k)\|_2,
%	\end{align*}
%	where the last inequality uses Lemma~\ref{sum_equ}. The lemma follows by plugging the bounds on~$s_1$ and~$s_2$ into Eq.~\eqref{inte_1}.
%\end{proof}

Next, we develop a relation for~$\|\mb{y}(k+1)-B_\infty\mb{y}(k+1)\|_B$.
\begin{lem} \label{3}
The following inequality holds,~$\forall k$:
	\begin{align}
		\|\mb{y}&(k+1)-B_\infty\mb{y}(k+1)\|_B\nonumber\\ 
		\leq&~ \sigma_B\beta lg\mn{A-I_{np}}_2\|\mb{x}(k)-A_\infty\mb{x}(k)\|_A\nonumber\\
		&+~\sigma_B\|\mb{y}(k)-B_\infty\mb{y}(k)\|_B +~\eta\sigma_B\beta l\|\mb{y}(k)\|_2.
	\end{align}
\end{lem}
\begin{proof}
We note that
	\begin{align}\label{inte_2}
		\|\mb{y}&(k+1)-B_\infty\mb{y}(k+1)\|_B \nonumber\\
		=&~\left\|B\Big(\mb{y}(k)+\nabla \mb{f}(k+1)-\nabla \mb{f}(k)\Big)-B_{\infty}B\Big(\mb{y}(k)+\nabla \mb{f}(k+1)-\nabla \mb{f}(k)\Big)\right\|_B, \nonumber\\
		\leq&~{\sigma_B\|\mb{y}(k)-B_\infty\mb{y}(k)\|_B+\sigma_B\beta l\|\mb{x}(k+1)-\mb{x}(k)\|_2,}
	\end{align}
because of Lemma~\ref{lem1}. Now we analyze $\|\mb{x}(k+1)-\mb{x}(k)\|_2$.
	\begin{align}\label{xd}
		\|\mb{x}&(k+1)-\mb{x}(k)\|_2 \nonumber\\
		=&~ \|A\mb{x}(k)-\eta \mb{y}(k)-\mb{x}(k)\|_2, \nonumber\\
		=&~ \left\|(A-I_{np})\big(\mb{x}(k)-A_\infty\mb{x}(k)\big)-\eta\mb{y}(k)\right\|_2,\nonumber\\
		\leq&~ \mn{A-I_{np}}_2g\|\mb{x}(k)-A_\infty\mb{x}(k)\|_A+\eta\|\mb{y}(k)\|_2.
	\end{align}
	The lemma follows by plugging Eq.~\eqref{xd} into Eq.~\eqref{inte_2}.
\end{proof}

The last step is to bound $\|\mb{y}(k)\|_2$ in terms of~$\|\mb{x}(k)-A_\infty\mb{x}(k)\|_A$,~$\|A_\infty\mb{x}(k)-\mb{1}_n \otimes \mb{x}^*\|_2$, and~$\|\mb{y}(k)-B_\infty\mb{y}(k)\|_B$. Then we can replace~$\|\mb{y}(k)\|_2$ in Lemma~\ref{1}-\ref{3} by this bound to complete the contraction relationship.
\begin{lem} \label{4}
The following inequality holds,~$\forall k$:
	\begin{align*}
		\|\mb{y}(k)\|_2 \leq&~ g\beta\mn{B_{\infty}}_2\|\mb{x}(k)-A_\infty\mb{x}(k)\|_A \nonumber\\
		&+~ \beta\mn{B_{\infty}}_2\|A_\infty\mb{x}(k)-\mb{1}_n \otimes \mb{x}^*\|_2 +~h\|\mb{y}(k)-B_\infty\mb{y}(k)\|_B.
	\end{align*}
\end{lem}
\begin{proof}
Recall that $B_{\infty}= (\bs{\pi}_c\otimes I_p)(\mb{1}_n^\top\otimes I_p).$ We have		
\begin{equation}\label{inte_3}
\|\mb{y}(k)\|_2 \leq h\|\mb{y}(k)-B_\infty\mb{y}(k)\|_B + \|B_\infty\mb{y}(k)\|_2.
\end{equation}
We next bound~$\|B_\infty\mb{y}(k)\|_2$:
\begin{align}\label{inte_4}
\|B_\infty\mb{y}(k)\|_2 =&~ \|(\bs{\pi}_c\otimes I_p)(\mb{1}_n^\top\otimes I_p)\mb{y}(k)\|_2\nonumber\\
=&~\|\bs{\pi}_c\|_2\|(\mb{1}_n^\top\otimes I_p)\nabla\mb{f}(k)\|_2 \nonumber\\
=&~ \|\bs{\pi}_c\|_2\left\|\sum_{i=1}^{n}\nabla f_i(\mb{x}_i(k))-\sum_{i=1}^{n}\nabla f_i(\mb{x}^*)\right\|_2 \nonumber\\
\leq&~ \|\bs{\pi}_c\|_2\beta\sum_{i=1}^{n}\|\mb{x}_i(k)-\mb{x}^*\|_2 \nonumber\\
\leq&~\|\bs{\pi}_c\|_2\beta\sqrt{n}\|\mb{x}(k)-\mb{1}_n\otimes \mb{x}^*\|_2, \nonumber\\
\leq&~\mn{B_{\infty}}_2\beta g\|\mb{x}(k)-A_\infty\mb{x}(k)\|_A +~\mn{B_{\infty}}_2\beta\|A_\infty\mb{x}(k)-\mb{1}_n \otimes \mb{x}^*\|_2,
\end{align}
where the second last inequality uses Jensen's inequality and the last inequality uses the fact that~$\mn{B_{\infty}}_2=\sqrt{n}\|\bs{\pi}_c\|_2$. The lemma follows by plugging Eqs.~\eqref{inte_4} into Eq.~\eqref{inte_3}.
\end{proof}

Before the main result, we present an additional lemma from nonnegative matrix theory.
\begin{comment}
\begin{lem}\label{pert}(Theorem 6.3.12 in~\cite{hornjohnson:13})
Let $A,E\in\mathbb{R}^{n\times n}$ and let~$q$ be a simple eigenvalue of $A$. Let $\mb{x}$ and $\mb{y}$ be, respectively, right and left eigenvectors of $A$ corresponding to $q$. Then
\begin{enumerate}[(i)]
	\item for each $\epsilon>0$, there exists a $\delta>0$ such that,~$\forall t\in\mathbb{C}$ with~$|t|<\delta$, there is a unique eigenvalue~$q(t)$ of $A+tE$ such that $\left|q(t)-q-t\frac{\mb{y}^*E\mb{x}}{\mb{y}^*\mb{x}}\right|\leq|t|\epsilon$, 
\item $q(t)$ is continuous at $t=0$, and $\lim_{t\rightarrow0}{q(t)}=q$.
\item $q(t)$ is differentiable $t=0$, 
$\frac{dq(t)}{dt}|_{t=0}=\frac{\mb{y}^*E\mb{x}}{\mb{y}^*\mb{x}}. $
\end{enumerate}
\end{lem} 
\end{comment}
\begin{lem}\label{rho}(Theorem 8.1.29 in~\cite{hornjohnson:13})
	Let $X\in\mathbb{R}^{n\times n}$ be a nonnegative matrix and~$\mb{x}\in\mathbb{R}^{n}$ be a positive vector. If~$X\mb{x}<\omega\mb{x}$, then~$\rho(X)<\omega$. 
\end{lem}
\subsection{Main results}
With the help of auxiliary relations developed in the previous subsection, we now present the main result, which establishes the geometric convergence of the proposed algorithm.
\begin{theorem}\label{thm1} If~$0<\eta<\frac{2}{n\beta\bs{\pi}_r^\top\bs{\pi}_c}$, we have the following linear matrix inequality (entry-wise):
	\begin{equation} \label{G}
		\mb{t}(k+1) \leq J(\eta)\mb{t}(k),~\forall k,
	\end{equation}
	where $\mb{t}(k)\in\mathbb{R}^3$ and $J(\eta)\in\mathbb{R}^{3\times3}$ are defined as follows: 
	\begin{align}\label{t,G}
	\mb{t}(k)&=\left[
	\begin{array}{l}
	\left\|\mb{x}(k)-A_\infty\mb{x}(k)\right\|_A \\
	\left\|A_\infty\mb{x}(k)-\mb{1}_n \otimes \mb{x}^*\right\|_2 \\
	\left\|\mb{y}(k)-B_\infty\mb{y}(k)\right\|_B
	\end{array}
	\right],\\
	J(\eta)&=\left[
	\begin{array}{ccc}
	\sigma_A+a_1\eta & a_2\eta &a_3\eta \\
	a_4\eta & \lambda & a_5\eta\\
	a_6+a_7\eta& a_8\eta & \sigma_B+a_{9}\eta
	\end{array}
	\right],
	\end{align}
	with the positive constants $a_i$'s being
	{\color{black}
	\begin{eqnarray*}
	a_1 &=& mg\beta\mn{I_{np}-A_\infty}_2\mn{B_{\infty}}_2, \nonumber\\
	a_2 &=& m\beta\mn{I_{np}-A_\infty}_2\mn{B_{\infty}}_2, \nonumber\\
	a_3 &=& mh\mn{I_{np}-A_\infty}_2, \nonumber\\
	a_4 &=& \begin{color}{black}n\end{color}\beta g(\bs{\pi}_r^\top\bs{\pi}_c), \nonumber\\
	a_5 &=& h\mn{A_{\infty}}_2,\nonumber\\
	a_6 &=& g\sigma_Bl\beta \mn{A-I_{np}}_2,\nonumber\\
	a_7 &=& g\sigma_Bl\beta^2\mn{B_{\infty}}_2,\nonumber\\
	a_8 &=& \sigma_Bl\beta^2 \mn{B_{\infty}}_2,\nonumber\\
	a_{9} &=& h\sigma_Bl\beta.
	\end{eqnarray*}
}When the step-size,~$\eta$, satisfies
\begin{align}
	\eta<\min\left\{\frac{\epsilon_1(1-\sigma_A)}{a_1\epsilon_1+a_2\epsilon_2+a_3\epsilon_3},~\frac{(1-\sigma_B)\epsilon_3-\epsilon_1a_6}{a_7\epsilon_1+a_8\epsilon_2+a_9\epsilon_3},~\frac{1}{n\beta\bs{\pi}_r^\top\bs{\pi}_c}\right\},
\end{align} 
%\begin{color}{red}\begin{align}
	%\eta<\min\left[\max\left\{\frac{\epsilon_1(1-\sigma_A)}{a_1\epsilon_1+a_2\epsilon_2+a_3\epsilon_3},~\frac{(1-\sigma_B)\epsilon_3-\epsilon_1a_6}{a_7\epsilon_1+a_8\epsilon_2+a_9\epsilon_3}\right\},~\frac{2}{n\beta}\right],
	%\end{align} \end{color}
where~$\epsilon_1,\epsilon_2,\epsilon_3$ are positive constants such that
\begin{align}
\epsilon_3> 0,\qquad\epsilon_1 < \frac{(1-\sigma_B)\epsilon_3}{a_6}, \qquad 
\epsilon_2 > \frac{a_4\epsilon_1+a_5\epsilon_3}{\alpha n(\bs{\pi}_r^\top\bs{\pi}_c)}, 
\end{align}
the spectral radius of $J(\eta)$,~$\rho(J(\eta))$, is strictly less than~$1$, and therefore~$\left\|\mb{x}(k)-\mb{1}_n \otimes \mb{x}^*\right\|_2$ converges to zero geometrically at the rate of~$O(\rho(J(\eta))^k)$.
\end{theorem}
\begin{proof}
Combining the results of Lemmas~\ref{1}--\ref{4}, one can verify that Eq.~\eqref{G} holds if~$0<\eta<\frac{2}{n\beta\bs{\pi}_r^\top\bs{\pi}_c}$. Recall that~$\lambda=\max\left(\left|1-\alpha n\eta(\bs{\pi}_r^\top\bs{\pi}_c)\right|,\left|1-\beta n\eta(\bs{\pi}_r^\top\bs{\pi}_c) \right|\right)$. When~$0<\eta<\frac{1}{n\beta\bs{\pi}_r^\top\bs{\pi}_c}$,~$\lambda=1-\alpha n\eta(\bs{\pi}_r^\top\bs{\pi}_c)$, since $\alpha\leq\beta$; see, e.g.,~\cite{opt_literature0} for details. 
\begin{comment}
We now split the matrix, $J(\eta)$, into the sum of a fixed matrix and another perturbation matrix as a function $\eta$:
	\begin{align}
		J(\eta) &= \left[
		\begin{array}{ccc}
		\sigma_A & 0 & 0\\
		0 & 1 & 0\\
		a_6 & 0 & \sigma_B
		\end{array}
		\right]
		+
		\eta\left[
		\begin{array}{ccc}
		a_1  & a_2 &a_3 \\
		a_4 & -\alpha n(\bs{\pi}_r^\top\bs{\pi}_c) & a_5\\
		a_7 & a_8 & a_{9}
		\end{array}
		\right] := J_0 + \eta E. \nonumber
	\end{align}
	Clearly, the spectral radius of $J_0$ is 1; recall that both~$\sigma_A$ and~$\sigma_B$ are in~$(0,1)$. It is straightforward to verify that the right and left eigenvector corresponding to the eigenvalue of~$1$ of~$J_0$ is~$\mb{v}=\left[0,1,0\right]^\top$. Denote by~$q(\eta)$, the eigenvalues of~$J(\eta)$ as a function of~$\eta$. From Lemma~\ref{pert}, since~$1$ is a simple eigenvalue of~$J(0)$,
\begin{equation}
\frac{d q(\eta)}{d\eta}\Bigg|_{\eta=0,q=1} = \frac{\mb{v}^\top E\mb{v}}{\mb{v}^\top\mb{v}} = -\alpha n(\bs{\pi}_r^\top\bs{\pi}_c),
\end{equation}
i.e.,~$\frac{d q}{d\eta}|_{\eta=0,q=1}<0$ and the spectral radius of~$J(\eta)$ is strictly less than $1$ as~$\eta$ slightly increases from zero. This is because the eigenvalues are continuous functions of the elements of the matrix. 
\end{comment}
The goal is to find an upper bound of the step-size,~$\widetilde{\eta}$, such that~$\rho(J(\eta))<1$ when~$\eta<\widetilde{\eta}.$ In the light of Lemma~\ref{rho}, we solve for the range of the step-size,~$\eta$, and a positive vector~$\bs{\epsilon}=\left[\epsilon_1,\epsilon_2,\epsilon_3\right]^\top$ 
from the following linear matrix inequality (entry-wise):
\begin{align}\label{eta1}
	\left[
	\begin{array}{ccc}
	\sigma_A+a_1\eta & a_2\eta &a_3\eta \\
	a_4\eta & 1-\alpha n\eta(\bs{\pi}_r^\top\bs{\pi}_c) & a_5\eta\\
	a_6+a_7\eta& a_8\eta & \sigma_B+a_{9}\eta
	\end{array}
	\right]
	\left[
	\begin{array}{ccc}
	\epsilon_1 \\
	\epsilon_2\\
	\epsilon_3
	\end{array}
	\right]
	<
	\left[
	\begin{array}{ccc}
	\epsilon_1 \\
	\epsilon_2\\
	\epsilon_3
	\end{array}
	\right],
\end{align}
which is equivalent to the following set of inequalities: 
\begin{align}
\left\{
\begin{array}{lll}
(a_1\epsilon_1+a_2\epsilon_2+a_3\epsilon_3)\eta&<&\epsilon_1(1-\sigma_A), \\
(a_4\epsilon_1-\alpha n(\bs{\pi}_r^\top\bs{\pi}_c)\epsilon_2+a_5\epsilon_3)\eta&<&0, \\
(a_7\epsilon_1+a_8\epsilon_2+a_9\epsilon_3)\eta&<&(1-\sigma_B)\epsilon_3-\epsilon_1a_6,
\end{array} \nonumber
\right.
\end{align}
Solving the inequalities above, we have that when
\begin{align}
\left\{
\begin{array}{lll}
\epsilon_1 &<& \frac{(1-\sigma_B)\epsilon_3}{a_6}, \\
\epsilon_2 &>& \frac{a_4\epsilon_1+a_5\epsilon_3}{\alpha n(\bs{\pi}_r^\top\bs{\pi}_c)}, \\
\epsilon_3 &>& 0, \\
\eta &<& \min\left\{\frac{\epsilon_1(1-\sigma_A)}{a_1\epsilon_1+a_2\epsilon_2+a_3\epsilon_3},\frac{(1-\sigma_B)\epsilon_3-\epsilon_1a_6}{a_7\epsilon_1+a_8\epsilon_2+a_9\epsilon_3}\right\},
\end{array} \nonumber
\right.
\end{align}
the inequality in Eq.~\eqref{eta1} holds and the Theorem follows.
\end{proof}

\section{Numerical Experiments}\label{s4}
We consider a binary classification problem in the distributed setting, where we use logistic loss function to train a linear classifier. Each agent~$i$ has access to~$m_i$ training data,~$(\mb{c}_{ij},y_{ij})\in\mathbb{R}^p\times\{-1,+1\}$, where~$\mb{c}_{ij}$ contains~$p$ features of the~$j$th training data at agent~$i$ and~$y_{ij}$ is the corresponding binary label. For privacy issues, agents do not share training data with each other. In order to use the entire data set for training, the network of agents cooperatively solves the following distributed logistic regression problem:
\begin{align}
\underset{\mb{w}\in\mbb{R}^p,b\in\mathbb{R}}{\operatorname{min}}F(\mb{w},b)
=\sum_{i=1}^n\sum_{j=1}^{m_i}\ln\left[1+\exp\left(-\left(\mb{w}^\top\mb{c}_{ij}+b\right)y_{ij}\right)\right] \nonumber
+\frac{\xi}{2}\|\mb{w}\|_2^2\nonumber,
\end{align}where the private function at each agent,~$i$, is given by:
\[
f_i(\mb{w},b)=\sum_{j=1}^{m_i}\ln\left[1+\exp\left(-\left(\mb{w}^\top\mb{c}_{ij}+b\right)y_{ij}\right)\right] 
+\frac{\xi}{2n}\|\mb{w}\|_2^2.
\]
In our setting,{\color{black}~$n=8$},~$p=5$. The feature vectors,~$\mb{c}_{ij}$'s, are Gaussian with zero mean and variance~$2$. The binary labels are randomly generated from standard Bernoulli distribution. We first compare the performance of the proposed algorithm in this paper, with ADD-OPT/Push-DIGing~\cite{add-opt,opdirect_nedicLinear}, FROST~\cite{xin2018fast}, and subgradient-push{\color{black}~\cite{opdirect_Tsianous,opdirect_Nedic}}, over the leftmost directed graph,~$\mc{G}_1$, shown in Fig.~\ref{g}. The simulation results are shown in the left figure in Fig.~\ref{s}. Next, we evaluate the proposed algorithm on the three different directed graphs,~$\mc{G}_1,\mc{G}_2,\mc{G}_3$, shown in Fig.~\ref{g}, where each graph to the right has a few more edges compared to the one on its left. The simulation results are shown in the right figure in Fig.~\ref{s}. In both cases, we plot the average of the residuals at each agent,{\color{black}~$\frac{1}{n}\sum_{i=1}^{n}\|\mb{x}_i(k)-\mb{x}^*\|_2$}. We note that the proposed linear algorithm achieves a  geometric (linear on the log-scale) convergence speed comparable to other fast algorithms over directed graphs but with less computation and communication. These simulations confirm the theoretical findings in this letter. 

\begin{figure}[!h]
\centering
\subfigure{\includegraphics[width=1.8in]{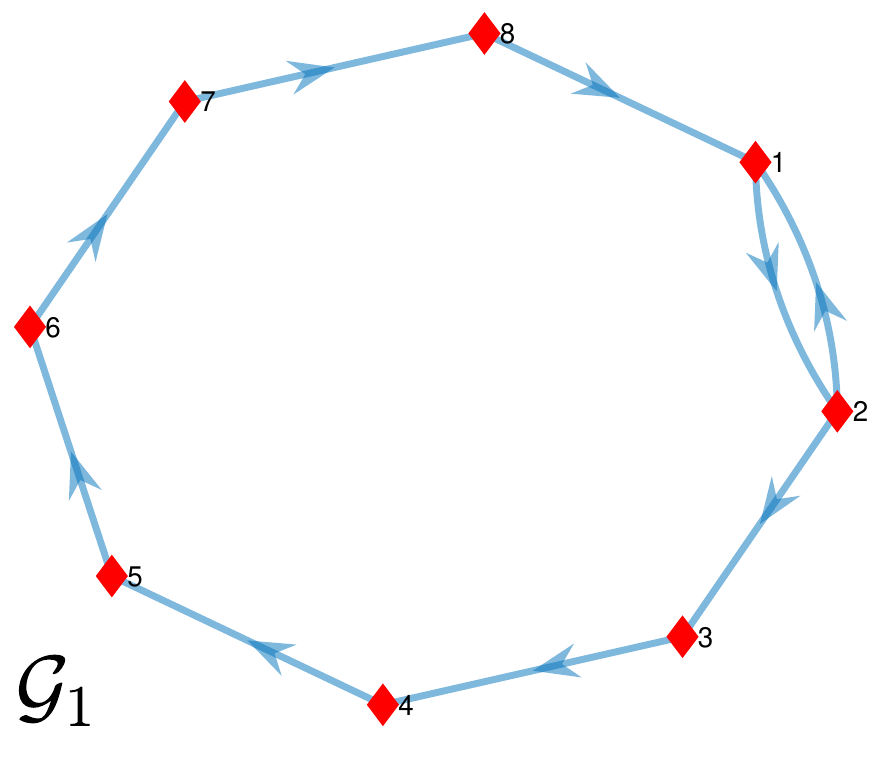}}
\hspace{0.5cm}
\subfigure{\includegraphics[width=1.8in]{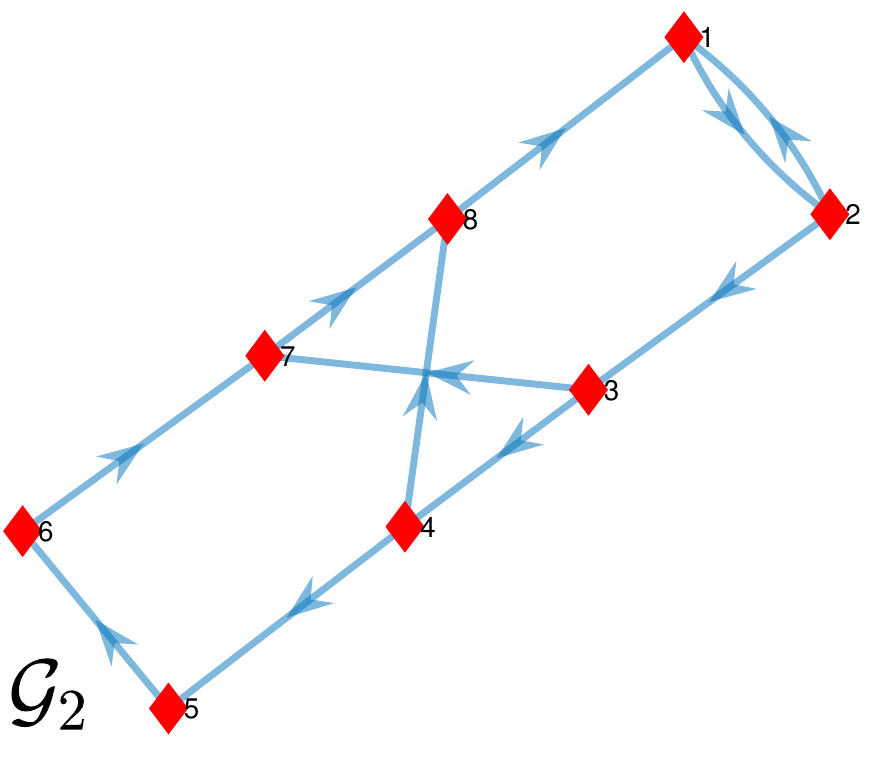}}
\subfigure{\includegraphics[width=1.8in]{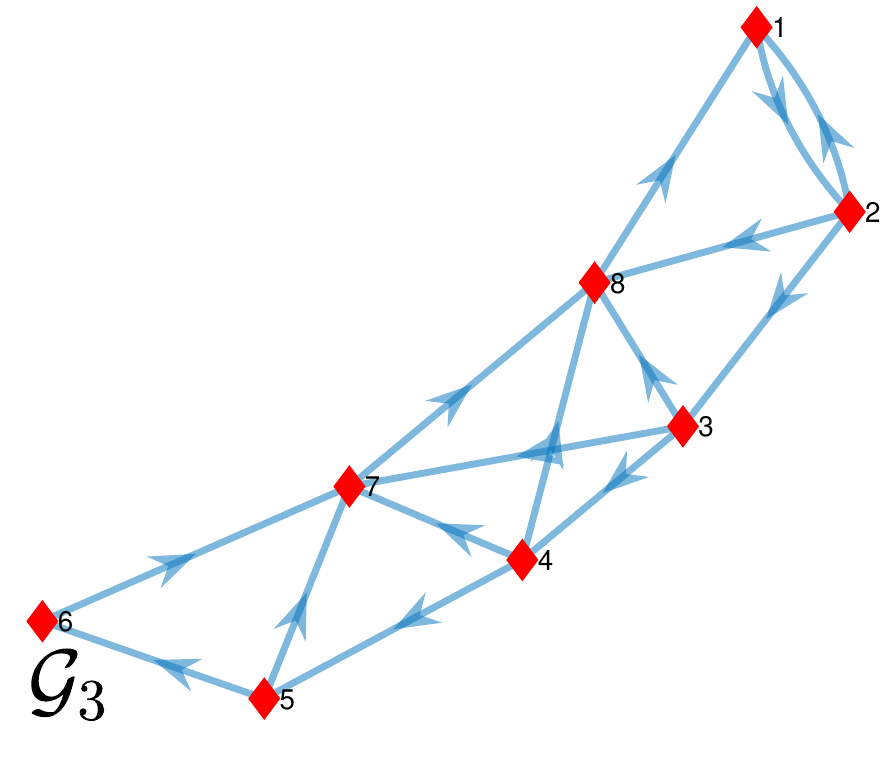}}
\caption{Strongly-connected but unbalanced directed graphs.}
\label{g}
\end{figure}

\begin{figure}[!h]
\centering
\subfigure{\includegraphics[width=2.5in]{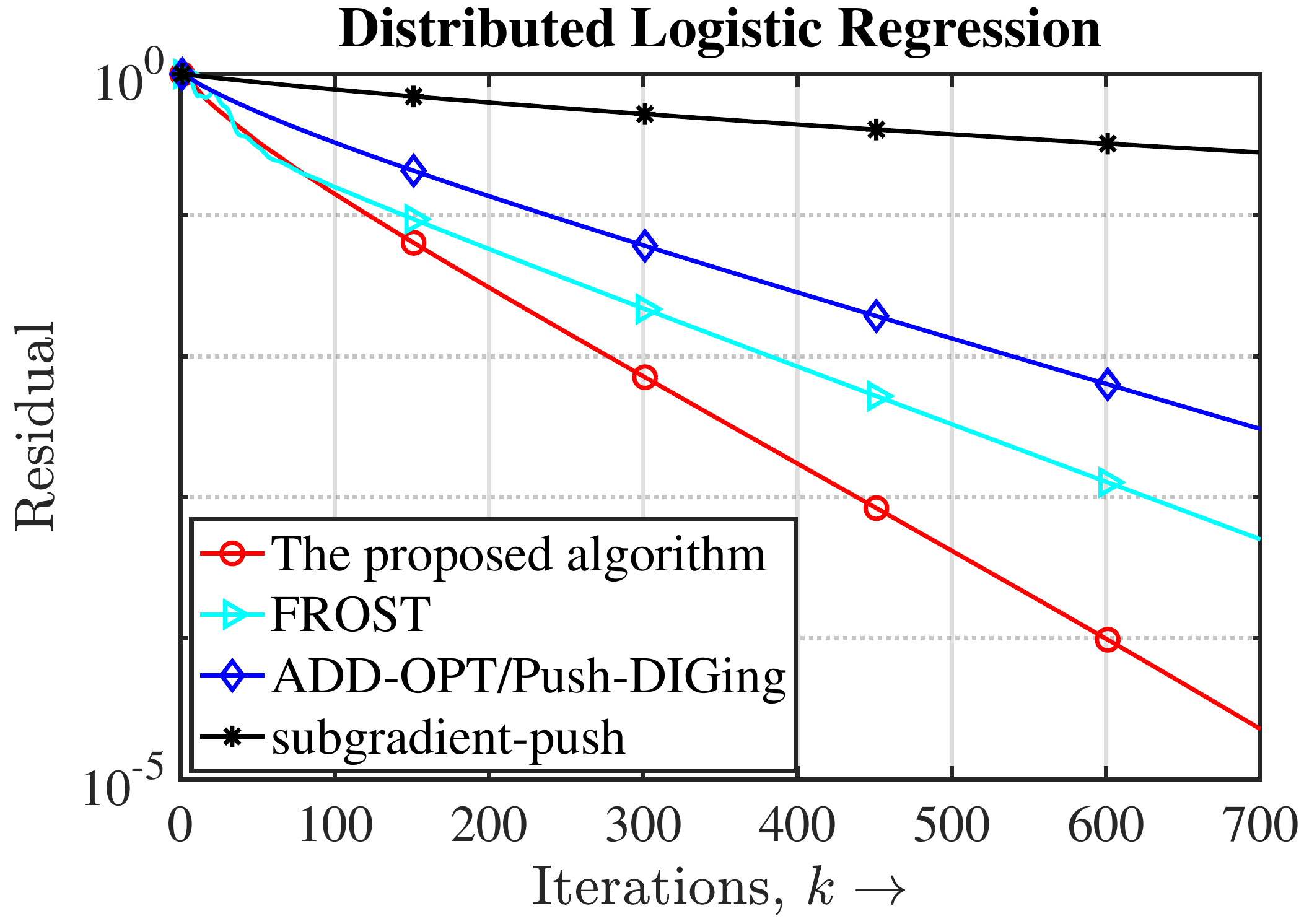}}
\hspace{0.5cm}
\subfigure{\includegraphics[width=2.5in]{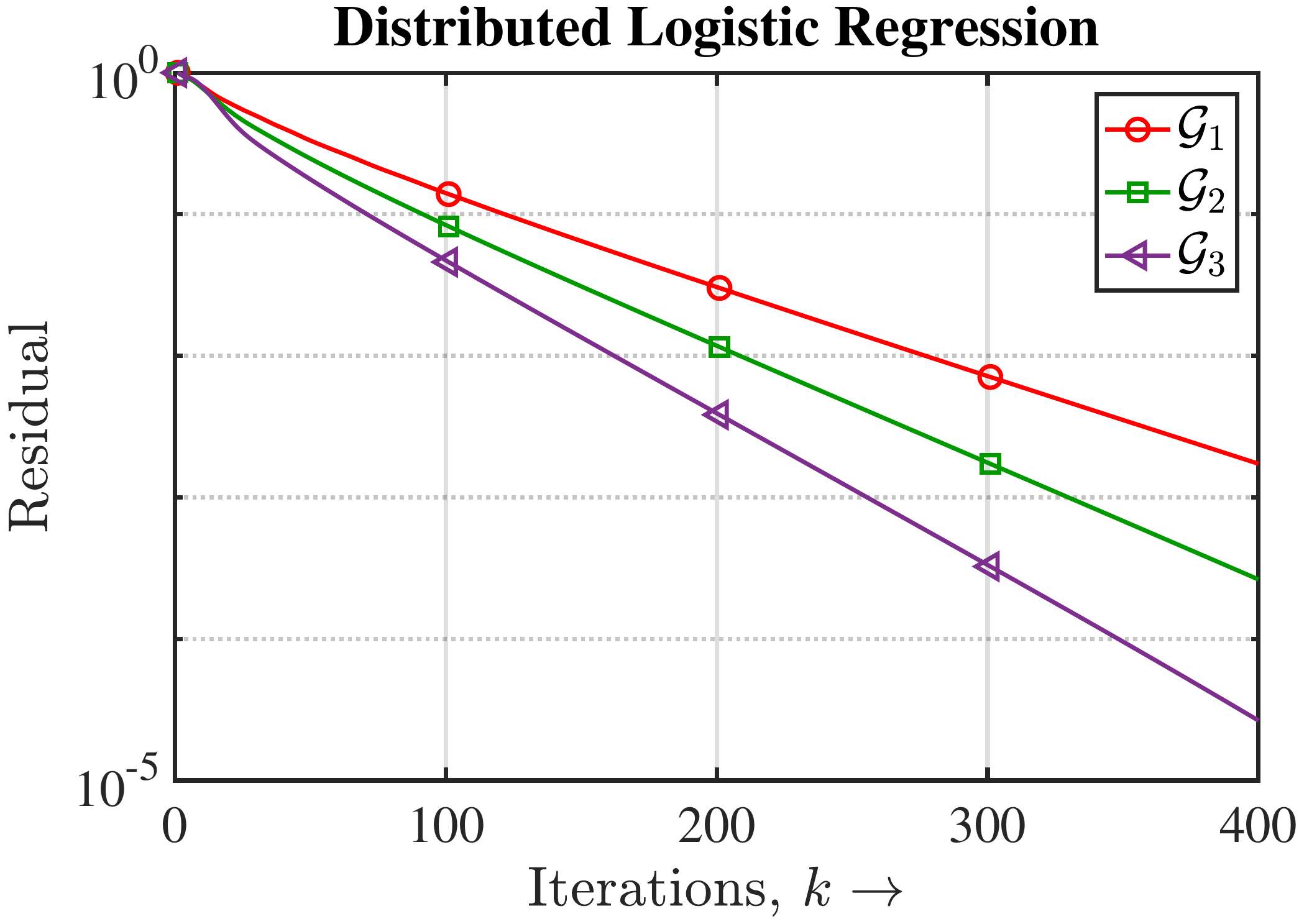}}
\caption{(Left) Comparison across different algorithms. (Right) Proposed algorithm over different graphs.we plot the average residuals at each agent,{\color{black}~$\frac{1}{n}\sum_{i=1}^{n}\|\mb{x}_i(k)-\mb{x}^*\|_2$}.}
\label{s}
\end{figure}

\section{Conclusions}\label{s5}
In this letter, we describe a linear distributed algorithm for optimization over directed graphs that can be seen as a generalization of earlier work over undirected graphs. Under the assumptions that the objective functions are strongly-convex and have Lipschitz-continuous gradients, the proposed algorithm achieves a geometric convergence to the global optimal. Our analysis is based on a novel approach where we establish simultaneous contractions of both row- and column-stochastic matrices under some arbitrary norms. We then use an elegant result from nonnegative matrix theory to develop the conditions for convergence.

\bibliographystyle{IEEEbib}
\bibliography{sample}

\end{document}